\documentclass[12pt]{amsart}
\usepackage[utf8]{inputenc}
\usepackage{amsmath}
\usepackage{amssymb}
\usepackage{mathtools}
\usepackage{amsthm}
\usepackage[margin=1.1in]{geometry}
\usepackage{csquotes}
\usepackage{xcolor}
\usepackage{hyperref}
\usepackage{enumerate}
\usepackage{comment}
\usepackage{cleveref}

\newtheorem{theorem}{Theorem}[section]

\newtheorem{lemma}[theorem]{Lemma}
\newtheorem{corollary}[theorem]{Corollary}
\newtheorem{definition}{Definition}
\newtheorem{proposition}[theorem]{Proposition}

\newtheorem{example}[theorem]{Example}

\newtheorem{remark}[theorem]{Remark}

\newcommand{\F}{{\mathbb F}}

\newcommand{\Z}{{\mathbb Z}}

\newcommand{\fq}{{\mathbb F}_{q}}

\newcommand{\PG}{\mathrm{PG}}

\newcommand{\ZZZ}{\mathcal{Z}}

\newcommand{\rad}{\mathrm{rad}_q}

\usepackage{setspace}
\title[$k$-Blocking Sets and Power Residue]{Blocking Sets and Power Residue Modulo Integers with bounded number of prime factors}
\author[B. Mishra]{Bhawesh Mishra} \author[P. Santonastaso]{Paolo Santonastaso}
\date{}

\begin{document}

\begin{abstract}
Let $q$ be an odd prime and $k$ be a natural number. We show that a finite subset of integers $S$ that does not contain any perfect $q^{th}$ power, contains a $q^{th}$ power residue modulo almost every natural numbers $N$ with at most $k$ prime factors if and only if $S$ corresponds to a $k$-blocking set of $\PG(\mathbb{F}_{q}^{n})$. Here, $n$ is the number of distinct primes that divides the $q$-free parts of elements of $S$. 
Consequently, this geometric connection enables us to utilize methods from Galois geometry to derive lower bounds for the cardinalities of such sets $S$ and to completely characterize such $S$ of the smallest and the second smallest cardinalities. Furthermore, the property of whether a finite subset of integers contains a $q^{th}$ power residue modulo almost every integer $N$ with at most $k$ prime factors is invariant under the action of projective general linear group $\mathrm{PGL}(n, q)$.
\end{abstract}

\maketitle

\noindent
\textbf{Keywords:} Prime Power Residue; Blocking Set; Local-to-Global Principle.\\
\textbf{MSC2020:}  51E21; 05B25; 11A15.

\section{Introduction}

\subsection{Motivation}
Let $q$ be a prime. We will say that a subset $S$ of integers contains a $q^{th}$ power modulo almost every prime $p$ if and only if for cofinitely many primes $p$, the congruence \[x^{q} \equiv s \hspace{1mm} (\text{mod } p)\] has a solution $x \in\mathbb{Z}$ and $s \in S$. If a set $S$ already contains an integer $q^{th}$ power, then it trivially contains a $q^{th}$ power modulo almost every prime. The interesting case is when $S$ does not contain an integer $q^{th}$ power and yet contains a $q^{th}$ power modulo almost every prime. Every such $S$ constitutes an instance of failure of the local-to-global principle in number theory (see \cite[pp. 99-108]{Gou} for more details on the local-to-global principle).

The study of finite subsets of integers that contain a $q^{th}$ power residue modulo almost every prime, has a long and fruitful history. For instance, Fried first obtained a characterization of such subsets $S$ for $q = 2$ in \cite{fried1969arithmetical}, and the same result later also appeared in a work of Filaseta and Richman \cite{filaseta1989sets}. The analogous result for general $q^{th}$ powers was obtained by A. Schinzel and M. Skałba in \cite{SS} and is combinatorially quite complex in nature. The result in \cite{SS} also deals with a more general problem over general number fields. We refer the readers to \cite[Theorems 1, 2]{SS} for the results obtained by Schinzel and Skałba. When $q$ is an odd prime power, the results from \cite{SS} were further simplified by Skałba in \cite{skalba2005sets}. 

In the case, when $q$ is an odd prime, Skałba's characterization was further refined in \cite{mishra2023prime} by showing that sets $S$ that contain a $q^{th}$ power modulo almost every prime are in correspondence with \textit{linear covering} of suitably defined vector spaces.  It is also worth mentioning that power residue problems have also been investigated over abelian varieties \cite{wong2000power} and elliptic curves \cite{skalba2004power}. Furthermore, this line of inquiry into power residues has recently yielded interesting and fruitful connections to the theory of intersective polynomials \cite{mishra2024polynomials} and in the generalization of the Grunwald-Wang theorem from one rational to subsets of rationals \cite{mishra2024generalization}.

In this article, we establish that finite subsets $S$ of integers that contain $q^{th}$ power residue modulo almost every integer of the form $p_{1}p_{2} \cdots p_{k}$ in a non-trivial way, are in correspondence with $k$-blocking sets in $\PG(\mathbb{F}_{q}^{n})$. Here, $n$ is the number of distinct primes that divide $q$-free parts of elements of $S$. This is the first instance known to the authors when a number-theoretic phenomenon, i.e. failure of a certain local-to-global principle for prime powers, is equivalent to another phenomenon in finite geometry, i.e. existence of blocking sets. This connection enables us to establish lower bounds for cardinality of such sets $S$, classify such $S$ for which $|S|$ is the lowest (and the second lowest) and specify an action under which the above property of a finite subset $S$ of integers, is invariant. We will first introduce some preliminary concepts and notations.

\subsection{Identifying Finite Subsets $S$ with $\mathbb{F}_{q}^{n}$}
\label{ProjectiveIdentification}
From now onwards, $q$ will always denote an odd prime. We will say that a rational $s$ is a perfect $q^{th}$ power when $s = r^{q}$ for some $r \in\mathbb{Q}$. The set of non-zero rationals will be denoted by $\mathbb{Q}^{\times}$. A positive integer $s$ will be called $q$-free when $p^{q} \nmid s$ for any prime $p$. Given a prime $p$ and an integer $n$, we will say that $p^{a} \mid\mid n$ for some $a \geq 0$, when $p^{a}$ is the highest power of $p$ that divides $n$. 

\subsubsection{Reduction to positive $q$-free numbers}
Let $S = \{s_{j}\}_{j=1}^{\ell}$ be a finite subset of integers, that does not contain any perfect $q^{th}$ power and we are interested in studying whether $S$ contains a $q^{th}$ power residue. Since $-1$ is always a perfect $q^{th}$ power, an integer $s$ is a $q^{th}$ power (modulo any integer $m$) if and only if $|s|$ is so. Therefore, as long as $q$ is odd, it suffices to study $\{\lvert s_{j} \rvert\}_{j=1}^{\ell}$ instead of $S$ itself. 

Given a positive integer $r$ with unique factorization $\prod_{i=1}^{\mu} p_{i}^{a_{i}}$, we define \[ \text{rad}_{q}(r) := \prod_{i=1}^{\mu} p_{i}^{a_{i} \hspace{1mm} (\text{mod } q)}, \] which is the $q$-free part of the natural number $r$. Note that, an integer $s$ is a $q^{th}$ power modulo a prime $p \nmid b$ if and only if the integer $s \cdot b^{q}$ is so too. Therefore, as long as we are concerned with $q^{th}$ power residue modulo cofinitely many primes, we can study the set  $\{\text{rad}_{q}\left(\lvert s_{j} \rvert\right)\}_{j=1}^{\ell}$ in place of $S$. This is because there are only finitely many primes that may divide some element in $S = \{s_{j}\}_{j=1}^{\ell}$ but do not divide any element in $\{\text{rad}_{q}\left(\lvert s_{j} \rvert\right)\}_{j=1}^{\ell}$.

\subsubsection{Identification with $\mathbb{F}_{q}^{n}$} We will use $\mathbb{F}_{q}$ to denote finite field with $q$ elements. For a vector space $V$ over $\mathbb{F}_{q}$, we will use $\PG(V)$ to denote the projective space generated by $V$. If $W$ is a $(k+1)$-dimensional subspace of $V$, then we will say that $\PG(W)$ is a $k$-space of $\PG(V)$. 

Given a finite subset $S = \{s_{j}\}_{j=1}^{\ell}$ of integers not containing any perfect $q^{th}$ power, let $p_{1} < p_{2} < \ldots < p_{n}$ be all the distinct primes that divides $\prod_{j=1}^{\ell}\text{rad}_{q}(|s_{j}|)$. For every $1 \leq j \leq \ell$ and every $1 \leq i \leq n$, let $a_{ij} \geq 0$ such that $p_{i}^{a_{ij}} \mid\mid \text{rad}_{q}(|a_{j}|)$. Then, we can identify every element of $S$ with an element in $\mathbb{F}_{q}^{n}$ through the map \[ \pi_{q} : S \longrightarrow  \mathbb{F}_{q}^{n}\setminus\{0\}, \] where $\pi_{q}(s_{j}) = \left( a_{ij} \right)_{i=1}^{n}$ for every $ 1 \leq j \leq \ell$. In this way, we can associate the set $S$ with a set of points \[ \Big\{ \langle (a_{11}, a_{21}, \ldots, a_{n1} )\rangle, \langle (a_{12}, a_{22}, \ldots, a_{n2} )\rangle, \ldots, \langle (a_{1\ell}, a_{2\ell}, \ldots, a_{n\ell} )\rangle \Big\} \subseteq\PG(\mathbb{F}_{q}^{n}), \] which we will call \textit{the set of projective points associated with $S$}.  

Furthermore, we say that a subset $T \subset S$ is \textit{$\mathbb{F}_{q}$-linearly independent} if and only if the set $\pi_{q}(T)$ is a $\mathbb{F}_{q}$-linearly independent subset of $\mathbb{F}_{q}^{n}$. We will say that a subset $T$ of integers generates a subspace $V$ of $\mathbb{F}_{q}^{n}$ when $\pi_{q}(T)$ generates the subspace $V$ of $\mathbb{F}_{q}^{n}$.

\subsection{Our contribution}
Rather than considering $q^{th}$ power residue modulo almost every prime only, in this paper we consider $q^{th}$ power modulo almost every integer that have at most $k$ (not necessarily distinct) prime factors. Let $\mathcal{P}_{f}(\mathbb{Z})$ denote the family of finite subsets of $\mathbb{Z}$. Given an odd prime $q$ and a natural number $k$, we define
\[ \mathcal{T}_{k, q} := \Big\{ S \in\mathcal{P}_{f}(\mathbb{Z}) : S \text{ contains a $q^{th}$ power modulo almost every integer $N$ with } \Omega(N) \leq k \Big\}. \]

The phrase \textit{almost every integer} and the quantity $\Omega(N)$ are made precise through the following definitions. 

\begin{definition}
Given a natural number $r > 1$ with unique prime factorization $r = \prod_{i=1}^{\mu} p_{i}^{a_{i}} $, we define the $\Omega(r)$ to be the quantity $\sum_{i=1}^{\mu} a_{i}$. In other words, $\Omega$ is the prime-factor counting function that honors multiplicities.
\end{definition}

\begin{definition}
Let $k$ be a natural number. We say that a finite set $S$ of integers contains a $q^{th}$ power modulo almost every natural number $N$ with $\Omega(N) \leq k$ when the following holds:
\begin{center}
There exists a natural number $\Delta$ (depending upon $S$) such that for every natural number $N$ with $\Omega(N) \leq k$ and $\text{gcd }(N, \Delta) = 1$, $S$ contains a $q^{th}$ power modulo $N$. \label{Delta}
\end{center}
\end{definition}

In this article, once the set $S = \{s_{j}\}_{j=1}^{\ell}$ of integers are fixed, the quantity $\Delta$ is of the form $(-q)^{\mu_{0}} \prod_{j=1}^{\ell} s_{j}^{\mu_{j}}$, where $\mu_{0}, \mu_{1}, \ldots, \mu_{\ell}$ are all natural numbers. Although not needed in this article, interested readers can consult \cite[pp. 5]{mishra2022polynomials} for an explicit formula for $\Delta$ in terms of $q$ and $\{s_{j}\}_{j=1}^{\ell}$ Therefore, the only exceptional natural numbers $N$ in the definition above are the ones that are either divisible by $q$ or share a common prime factors with $s_{j}$'s. This leads to the following remark. 

\begin{remark}
In light of the comment above, for $k = 1$ the phrase \textbf{almost every prime} is equivalent to \textbf{cofinitely many primes}. This is because the exceptional primes are the ones that divide $q\prod_{j=1}^{\ell} s_{j}$. However, the phrase \textbf{almost every} does not imply \textbf{cofinitely many} in the case $k > 1$ due to the fact that there are infinitely many integers $N$ with $\Omega(N) \leq k$ that can share a prime factor with $q\prod_{j=1}^{\ell} s_{j}$. 
Furthermore, for $k = 1$, $S \in \mathcal{T}_{1,q}$ is equivalent to $S$ containing a $q^{th}$ for all powers of prime $p$, except for co-finitely exceptional primes $p$. Even for the exceptional primes $p$, one only needs to check up to a finite power $p^{k}$ that depends upon $S$. Both of the previous sentences are a consequence of the Hensel's lemma for the polynomial $\prod_{s \in S} (x^{q} - s)$. Therefore, the crucial difference between $S \in\mathcal{T}_{1,q}$ versus $S \in \mathcal{T}_{2,q}$ comes down to a set $S$ containing a $q^{th}$ power modulo $p_{1}p_{2}$ for all distinct primes $p_{1}$ and $p_{2}$. So, $S$ being in $\mathcal{T}_{k,q}$ is a stronger condition than $S \in \mathcal{T}_{1,q}$ as studied in \cite{SS, skalba2005sets}.
\end{remark}

The sets in the family $\mathcal{T}_{k,q}$ were first studied in \cite{skalba2004alternatives} by Skałba, primarily with the goal of studying the lower bound on the cardinality of sets $S \in\mathcal{T}_{k,q}$. We extend this line of inquiry with a Galois-geometric characterization of sets $S\in\mathcal{T}_{k,q}$ which leads to a host of other structural and classification results that are not available otherwise. 

\begin{definition}
Let $n > k \geq 1$. A subset $\mathcal{S}\subseteq\PG(\mathbb{F}_{q}^{n})$ is said to be a \textbf{$k$-blocking set} if given every subspace $W$ of $\mathbb{F}_{q}^{n}$ with codimension $k$, one has $\PG(W) \cap \mathcal{S} \neq\emptyset$.
\label{defnblocking} 
\end{definition}
The main result in this article is the following correspondence between the elements of $\mathcal{T}_{k,q}$ and $k$-blocking sets. 

\begin{theorem} 
Let $q$ be an odd prime and $k$ be a natural number. Let $S = \{s_{j}\}_{j=1}^{\ell}$ be a finite subset of integers not containing any perfect $q^{th}$ power and $n$ be the number of distinct primes that divide $\prod_{j=1}^{\ell} \text{rad}_{q}(\lvert s_{j} \rvert)$ Then, the following two statements are equivalent:
\begin{enumerate}
    \item The set $S$ belongs to the collection $\mathcal{T}_{k,q}$. 

    \item The set of projective points associated with $S$ is a $k$-blocking set of $\PG(\mathbb{F}_{q}^{n})$. 
\end{enumerate}
\label{mainresult}
\end{theorem}

This connection with $k$-blocking sets allows us to employ techniques from Galois geometry to investigate sets in $\mathcal{T}_{k,q}$. More precisely, we prove that the property of whether a finite subset $S$ of integers contains a $q^{th}$ power residue modulo almost every integer $N$ with at most $k$ prime factors is invariant under the action of projective general linear group $\mathrm{PGL}(n, q)$. Moreover, we $(i)$ establish lower bounds on cardinalities of sets in $\mathcal{T}_{k,q}$, $(ii)$ characterize sets in $\mathcal{T}_{k,q}$ achieving this lower bound and $(iii)$ construct some minimal sets in $\mathcal{T}_{k,q}$ of second smallest size for every odd prime $q$ and every $k \geq 2$.

\section{Some Preliminary Results} \ref{mainresult} 
Before we dive into the proofs, we will need the power residue symbol and some of its elementary properties. Let $K$ be a number field that contains the complex $q^{th}$ root of unity $\zeta_{q}$ and $\mathcal{O}_{K}$ be its ring of integers. Then, for every prime ideal $\mathfrak{p}$ of $K$ coprime to $q\mathcal{O}_{K}$ and every $\mathfrak{p}$-adic unit $\alpha \in K$, we define the $q^{th}$ power residue symbol $\left(\frac{\alpha}{\mathfrak{p}}\right)_{q}$ to be the unique $q^{th}$ root of unity $\zeta_{q}^{j}$ such that 
\[\alpha^{\frac{\text{Norm}(\mathfrak{p})-1}{q}} \equiv \zeta_{q}^{j} \hspace{1mm} (\text{mod } \mathfrak{p}).\] Whenever $\alpha, \beta$ are two $\mathfrak{p}$-adic unit, the power residue symbol obeys the multiplicative relation \[  \left(\frac{\alpha\beta}{\mathfrak{p}}\right)_{q} = \left(\frac{\alpha}{\mathfrak{p}}\right)_{q} \left(\frac{\beta}{\mathfrak{p}}\right)_{q}. \] We extend the power residue symbol for a non-prime ideal as follows: if $\mathfrak{a} = \mathfrak{p}_{1} \cdot \mathfrak{p}_{2} \cdots \mathfrak{p}_{s}$, we define \[  \left(\frac{\alpha}{\mathfrak{a}}\right)_{q} = \prod_{i=1}^{s}  \left(\frac{\alpha}{\mathfrak{p}_{i}}\right)_{q}  \text{ for every } \alpha \text{ coprime to } \mathfrak{a}. \] 

\begin{remark}
  Given a prime $\mathfrak{p}$ of $K$, $p = \mathfrak{p} \cap \mathbb{Z}$ is a prime in $\mathbb{Z}$. The (inertial) degree of $\mathfrak{p}$ is defined to be the index of $\mathbb{Z}/p\mathbb{Z}$ in $\mathcal{O}_{K}/\mathfrak{p}\mathcal{O}_{K}$. Let $\mathfrak{p}$ be a prime of $K$ of degree $1$ and $\alpha$ be an element of $K$ such that the power residue symbol $\left( \frac{\alpha}{\mathfrak{p}} \right)_{q}$ is defined. If $\left( \frac{\alpha}{\mathfrak{p}} \right)_{q} \neq 1$, then $\alpha$ is not a $q^{th}$ power modulo $\mathfrak{p}$ in $K$ and hence $\alpha$ is not a $q^{th}$ power modulo $p$ in $\mathbb{Q}$ either. We will repeatedly make use of this fact in the proof of Theorem \ref{mainresult}.
\end{remark}

For analogous reasons as in \ref{ProjectiveIdentification}, it suffices to assume that the elements of $S$ are positive and $q$-free so that for every $s \in S$, $\text{rad}_{q}(\lvert s \rvert) = s$. This is because modifying the elements of $S$ by perfect $q^{th}$ powers does not change its membership in the collection $\mathcal{T}_{k,q}$. Recall that in the Theorem \ref{mainresult}, we claim that the set of projective points associated with $S$ form a $k$-blocking set of $\PG(\mathbb{F}_{q}^{n})$ - a statement that is meaningful only if $n \geq (k +1)$. So, we will first establish that $n \geq k + 1$ through the following proposition. 

\begin{proposition}
Let $q$ be an odd prime and $S$ be a finite subset of integers not containing a perfect $q^{th}$ power. Suppose that there exists a natural number $k \geq 1$ such that $S \in\mathcal{T}_{k,q}$. Then, the number of distinct primes that divide $\prod_{s\in S} \text{rad}_{q}(\lvert s\rvert)$ is at least $(k + 1)$. \label{numberofprimes}
\end{proposition}

\begin{proof}
We will establish the proposition through induction on $k$. For the base case of $k = 1$, we refer the reader to \cite[pp. 5]{mishra2023prime}, which explains that for any prime $p$, the set $\{p, p^{2}, \ldots, p^{q-1}\}$ fails to have a $q^{th}$ power modulo infinitely many primes. In other words, at least two primes must divide $\prod_{s\in S} \text{rad}_{q}(|s|)$.

Now, let us assume that the proposition holds for all natural numbers $\leq k$ and suppose that $S$ is a finite subset of integers that does not contain a perfect $q^{th}$ power, but contains a $q^{th}$ power modulo almost every natural number N with $\Omega(N)$ $\leq k + 1$. 

For the sake of contradiction, assume that the number of distinct primes that divide elements of $S$ is at most $(k + 1)$, say $p_{1}, p_{2}, \ldots, p_{\mu}$ for some $\mu \leq (k+1)$. If $\mu \leq k$, the proposition follows from the inductive case. So, we assume that $\mu = k + 1$ without loss of generality. In this case, $S \subseteq ( R \times R^{\prime} ) \setminus\{1\}$, where \[ R = \left\{ \prod_{i=1}^{k} p_{i}^{a_{i}} : (a_{i})_{i=1}^{k} \in\mathbb{F}_{q}^{k} \right\} \text{ and } R^{\prime} = \{ p_{k+1}^{a_{k+1}} : a_{k+1} \in\mathbb{F}_{q} \}. \] By inductive hypothesis, we have the following: 
\begin{enumerate}
    \item For every $\Delta$, there exist infinitely many primes $p$ (i.e., $N$ with $\Omega(N) = 1$) with $p \nmid \Delta$ such that $R^{\prime}\setminus\{1\}$ does not contain a $q^{th}$ power modulo $p$. Hence, $R^{\prime}\setminus\{1\}$ does not contain a $q^{th}$ power modulo $p^{2}, p^{3}, \ldots, p^{k+1}$ either, when $p$ is one of these primes. 

    \item For every $\Delta$, there exists $N_{1}$ with $\Omega(N_{1}) \leq k$ and gcd$(N_{1}, \Delta) = 1$ such that $R\setminus\{1\}$ contains no $q^{th}$ power modulo $N_{1}$. 

    \item There exists $\Delta_{0}$ such that for every $N$ with gcd$(N, \Delta_{0}) = 1$ and $\Omega(N) \leq k - 1$, $R\setminus\{1\}$ contains a $q^{th}$ power modulo $N$. More specifically, for every prime $p$ with $p \nmid \Delta_{0}$, $R \setminus\{1\}$ contains a $q^{th}$ power modulo $p^{k-1}$.
\end{enumerate}
$(1)$ and $(2)$ above are a result of contrapositive of inductive hypothesis applied to $R^{\prime}\setminus\{1\}$ and $R\setminus\{1\}$ respectively, whereas $(3)$ is obtained from inductive hypothesis applied to $R\setminus\{1\}$. Note that every element of $S$ is of the form $s = rr^{\prime}$ where $r \in R$, $r^{\prime}\in R^{\prime}$ and at least one of $r, r^{\prime}$ is not equal to $1$. Let $\Delta$ be a natural number. Three cases arise:
\begin{itemize}
    \item \textit{Case $1$: $r^{\prime} = 1$}: In this case, we choose $N_{1}$ from $(2)$ above, which gives 
    \[ \left(\frac{s}{N_{1}}\right)_{q} = \left( \frac{r}{N_{1}} \right)_{q} \neq 1. \]

    \item \textit{Case $2$: $r = 1$}: In this case, we choose $p$ from $(1)$ above, which gives 
    \[ \left(\frac{s}{p^{k+1}}\right)_{q} = \left( \frac{s_{1}}{p^{k+1}} \right)_{q} \neq 1. \].

    \item \textit{Case $3$: $r \neq 1 \neq r^{\prime}$} In this case, we choose a prime $p$ from $(1)$ that is coprime to $\Delta_{0}$ in $(3)$, which gives

    \[ \left( \frac{s}{p^{k-1}} \right)_{q} = \left( \frac{r}{p^{k-1}} \right)_{q} \left( \frac{r^{\prime}}{p^{k-1}} \right)_{q} = 1 \left( \frac{r^{\prime}}{p^{k-1}} \right)_{q} \neq 1.  \]
\end{itemize}
Regardless of cases above, we have shown that for every $\Delta$, there exists $N$ with $\Omega(N) \leq k + 1$ such that $S$ does not contain a $q^{th}$ power modulo $N$, which establishes the proposition.
\end{proof}

In order to proceed with the proof of Theorem \ref{mainresult}, we will use the following fundamental result taken from \cite[Theorem 7.40, pp. 380]{Narkiewicz}.

\begin{proposition}
Let $q$ be a fixed rational prime, and let $K$ be an algebraic
number field containing all the $q^{th}$ roots of unity. Let $a_{1}, a_{2}, \ldots, a_{m}$ be finitely many elements in the ring of integers of $K$ that form a $\mathbb{F}_{q}$-linearly independent set, and let $z_{1}, z_{2}, \ldots, z_{m}$ be $q^{th}$ roots of unity. 
Then, there exist infinitely many unramified prime ideals $\mathfrak{p}$ of degree $1$ over $\mathbb{Q}$ such that for every $i \in\{1, 2, \ldots, m\}$, $\left(\frac{a_{i}}{p}\right)_{q} = z_{i}$.\label{characterresidue}
\end{proposition}

The proposition above gives the following lemma, which in our concrete context is a key ingredient in the proof of our main result. 

\begin{lemma}
\label{characterlemma}
Let $S = \{s_{j}\}_{j=1}^{\ell}$ be a finite subset of integers and let $n$ be the number of distinct primes that divide $\prod_{j=1}^{\ell} \text{rad}_{q}(\lvert s_{j} \rvert)$ so that the subspace $V$ generated by the set $\pi_{q}(S)$ is a subset of $\mathbb{F}_{q}^{n}$ as in \ref{ProjectiveIdentification}. Then, for every $\chi\in\hat{V}$, there exist infinitely many unramified primes $\mathfrak{p}$, of degree one, in $\mathbb{Q}(\zeta_{q})$ such that $\chi(v) = \left(\frac{\pi_{q}^{-1}(v)}{\mathfrak{p}}\right)_{q}$ for every $v \in V$. 
\end{lemma}

\begin{proof}
Let $K = \mathbb{Q}(\zeta_{q})$, $\mathcal{A} = \{a_{i}\}_{i=1}^{m}$ be a basis of $V$ and let $s_{i} := \pi_{q}^{-1}(a_{i})$ for every $1 \leq i \leq m$ (after reordering $s_{j}$'s if needed). Since $\mathcal{A}$ is also $\mathbb{F}_{q}$-linearly independent, the set $\pi_{q}^{-1}(\mathcal{A}) = \{s_{1}, s_{2}, \ldots, s_{m}\}$ is $\mathbb{F}_{q}$-linearly independent in $S$ by definition. An application of Proposition \ref{characterresidue} for $z_{i} = \chi(a_{i})$ implies that there exist infinitely many unramified prime ideals $\mathfrak{p}$ in $K$, of degree one, such that $\chi(a_{i}) = \left( \frac{s_{i}}{\mathfrak{p}} \right)_{q}$ for every $i = 1, 2, \ldots, m$.  

Let $v \in V$. Since the set $\mathcal{A}$ forms a basis for $V$, there exists $(c_{i})_{i=1}^{m} \in\mathbb{F}_{q}^{m}$ such that $v = \sum_{i=1}^{m} c_{i} a_{i}$, and hence $\pi_{q}^{-1}(v) = \prod_{i=1}^{m} s_{i}^{c_{i}}$ Therefore, we have
\begin{equation*}
     \chi(v) = \chi \left( \sum_{i=1}^{m} c_{i} a_{i} \right) =  \prod_{i=1}^{m} \chi ( c_{i} a_{i} )  = \prod_{i=1}^{m}  \chi (  a_{i} )^{c_{i}}  = \prod_{i=1}^{m} \left(\frac{s_{i}}{\mathfrak{p}}\right)^{c_{i}}_{q} = \left( \frac{\prod_{i=1}^{m} s_{i}^{c_{i}}}{\mathfrak{p}}\right)_{q} = \left( \frac{\pi_{q}^{-1}(v)}{\mathfrak{p}} \right)_{q},  
\end{equation*}
for any of such unramified, degree one, prime ideals $\mathfrak{p}$ in $K$. In the above series of equalities, the additive notation turns into a multiplicative notation because $\chi$ is a homomorphism from the additive group $\mathbb{F}_{q}^{n}$ to $\mathbb{C}^{\times}$.
\end{proof}

\begin{remark}
Note that Lemma \ref{characterlemma} above can also be obtained solely through the use of the Chebotarev's density theorem (see \cite[Theorem 7.30, pp. 368]{Narkiewicz}) for the field extension $K / \mathbb{Q}$, where $K = \mathbb{Q}\big(\zeta_{q}, a_{1}^{1/q}, a_{2}^{1/q}, \ldots, a_{m}^{1/q}\big)$ and $\{a_{i}\}_{i=1}^{m}$ are as in the proof of Lemma \ref{characterlemma}. This is because $\{a_{i}\}_{i=1}^{m}$ is a $\mathbb{F}_{q}$-linearly independent set and hence the Galois group of $K/\mathbb{Q}$ is isomorphic to the semi-direct product $\big(\mathbb{Z}/q\mathbb{Z}\big)^{m} \rtimes \big(\mathbb{Z}/q\mathbb{Z}\big)^{\times}$. However, this essentially amounts to reproving the Proposition \ref{characterresidue} using the same argument of the proof as in \cite[pp. 380]{Narkiewicz}. Similarly, Proposition \ref{numberofprimes} can also be obtained using Chebotarev density theorem; however, we choose to present a more elementary inductive proof. 
\end{remark}

Now, we are ready to establish our main result. 

\section{Proof of \Cref{mainresult}.}
\subsection{Proof of (1) implies (2)}
Assume that $S$ contains a $q^{th}$ power modulo almost every natural number $N$ with $\Omega(N) \leq k$ and $U$ be a subspace of $\mathbb{F}_{q}^{n}$ of codimension $k$. Such a subspace $U$ is defined by elements $\chi_{1}, \chi_{2}, \ldots, \chi_{k} \in \hat{\mathbb{F}_{q}^{n}}$. In other words, \[  U = \Big\{ v \in\mathbb{F}_{q}^{n} : \bigcap_{i=1}^{k} \chi_{i}(v) = 1 \Big\}. \]

Using Lemma \ref{characterlemma}, we have that for every $1 \leq i \leq k$, there exists infinitely many unramified primes $\mathfrak{p_{i}}$'s in $\mathbb{Q}(\zeta_{q})$ such that \[\chi_{j}(v) = \left(\frac{\pi_{q}^{-1}(v)}{\mathfrak{p_{i}}}\right)_{q} \text{ for every } v \in V.\] 
Since $S$ contains a $q^{th}$ power modulo almost every natural number $N$ with $\Omega(N) \leq k$, there must exist $\mathfrak{p}_{1}, \mathfrak{p}_{2}, \ldots, \mathfrak{p}_{k}$ and $s \in S$ such that $\left(\frac{s}{\mathfrak{p_{i}}}\right)_{q} = 1$ for each $1 \leq i \leq k$. Since $\pi_{q}(s) \in V$, we have $\chi_{i}\big(\pi_{q}(s)\big) = \left(\frac{\pi_{q}^{-1}(\pi_{q}(s)}{\mathfrak{p_{i}}}\right)_{q} = \left(\frac{s}{\mathfrak{p_{i}}}\right)_{q} = 1$ for every $1 \leq i \leq k$. Since the point sets $\mathcal{S} \subseteq \PG(\F_q^n)$ associated with $S$ contains the projective points defined by the elements of $\pi_{q}(S)$, we get that $\mathcal{S} \cap \PG(U) \neq \emptyset$. Therefore, $\mathcal{S}$ is a $k$-blocking set of $\PG(\F_q^n)$.

\subsection{Proof of $(2)$ implies $(1)$.}
Assume that the point set  $\mathcal{S} \subseteq \PG(\F_q^n)$ associated with $S$ is a $k$-blocking set of $\PG(\F_q^n)$. Therefore, $\pi_{q}(S)$ intersects every subspace of $U$ of $\mathbb{F}_{q}^{n}$ with codimension $k$ non-trivially. Since the projective points of $S$ are defined by the elements of $\pi_{q}(S)$, we get that $\pi_{q}(S)$ intersects every subspace of $U$ of $\mathbb{F}_{q}^{n}$ with codimension $k$ non-trivially. Furthermore, let $N$ be a natural number with $\Omega(N) \leq k$, i.e., $N = \prod_{i=1}^{\nu} p_{i}^{b_{i}}$ with $b_{i} \geq 1$, $\sum_{i=1}^{\nu} b_{i} \leq k$ such that $q \nmid N$. 

Define $\chi^{\prime}_{i}(v) := \left(\frac{\pi_{q}^{-1}(v)}{p_{i}} \right)_{q}$ for every $1 \leq i \leq \nu$. Since each of the $\chi^{\prime}_{i}$ are elements of $\hat{V}$, the subspace \[ U^{\prime} := \Big\{ v \in\mathbb{F}_{q}^{n} : \bigcap_{i=1}^{\nu} \chi^{\prime}_{i}(v) = \left(\frac{\cdot}{p_{i}} \right)_{q} = 1 \Big\}\] is a subspace of codimension $\nu \leq k$. Therefore, by assumption, there exists a $s \in S$ such that $\pi_{q}(s) \in U^{\prime}$, i.e., $\chi^{\prime}_{i}(\pi_{q}(s)) = \left(\frac{\pi_{q}^{-1}(\pi_{q}(s))}{p_{i}} \right)_{q} = \left(\frac{s}{p_{i}} \right)_{q} = 1$ for every $1 \leq i \leq \nu$. The proof works for any arbitrary natural number $N$ with $\Omega(N) \leq k$ that is coprime to $q \prod_{s\in S} s$, when the $q^{th}$ power residue symbol $\left(\frac{v}{p}\right)_{q}$ is defined. \qed

Before diving into some deeper structural consequences, we explore some immediate corollaries of Theorem \ref{mainresult}. First, the property of a set $S$ of belonging to the family $\mathcal{T}_{k,q}$ is invariant under element-wise exponentiation by elements of $\mathbb{F}_{q}^{\times}$. Furthermore, whether a finite set $S \subset\mathbb{Z}$ belongs to $\mathcal{T}_{k,q}$ depends only upon the factorization shape of its elements, and not on the specific primes that divide its elements. Both of these consequences, stated in the corollary below, follow because the projective points associated with $S$ neither change under exponentiation of elements of $S$ by elements of $\mathbb{F}_{q}^{\times}$ nor change under switching of primes.

\begin{corollary}\label{cor:indendenceprimes}
    \begin{enumerate}
        \item $\left( \text{Invariance Under Exponentiation} \right)$ For every $S = \{s_{j}\}_{j=1}^{\ell} \subset\mathbb{Z}$ and every $a_1,\ldots,a_{\ell} \in \F_q^{\times}$, we have that $S \in \mathcal{T}_{k,q}$ if and only if  $\{s_{j}^{a_j}\}_{j=1}^{\ell} \in \mathcal{T}_{k,q}$. \vspace{2mm}

        \item $\left( \text{Switching of Primes} \right)$  Let $\{p_{i}\}_{i=1}^{n}, \{\overline{p}_{i}\}_{i=1}^{n}$ be two distinct finite sets of rational primes and $S = \left\{\prod_{i=1}^{n} p_{i}^{\nu_{ij}}\right\}_{j=1}^{\ell}$. Then, $S \in \mathcal{T}_{k,q}$ if and only if 
    \[\left\{\prod_{i=1}^n\overline{p}_{i}^{\nu_{ij}}\right\}_{j=1}^{\ell} \in \mathcal{T}_{k,q}.\] \label{secondbasiccorollary}
    \end{enumerate}   
\end{corollary}

\section{Lower bounds and characterization of sets in $\mathcal{T}_{k,q}$}
In this section, we first provide lower bounds on the size of the sets in $\mathcal{T}_{k,q}$ that do not contain a perfect $q^{th}$ power. Then, we will also characterize those sets in $\mathcal{T}_{k,q}$ that have the minimum size. Finally, we will construct minimal sets in $\mathcal{T}_{k,q}$ of the second smallest cardinality. 

One of the main consequences of Theorem \ref{mainresult} is that the property whether a given $S$ belongs to $\mathcal{T}_{k,q}$ is invariant under a suitably defined action by elements of $\mathrm{PGL}(n,q)$, which we shall call \emph{geometric $q$-equivalence} - which is defined below.

\begin{definition} \label{def:geometricqequiv}
Let $S = \{s_{j}\}_{j=1}^{m}$ and $T = \{t_{j}\}_{j=1}^{\ell}$ be two finite sets of non-zero integers, not containing a perfect $q^{th}$ power. Let $p_{1}, p_{2}, \ldots, p_{n}$ be all the primes that divide $\Big( \prod_{j=1}^{m} \text{rad}_{q}(\lvert s_{j} \rvert) \times \prod_{j=1}^{\ell}  \text{rad}_{q}(\lvert t_{j} \rvert) \Big)$. Let $ \text{rad}_{q}(|s_{j}|) = \prod_{i=1}^{n} p_{i}^{\nu_{ij}}$
for every $j\in\{1, 2, \ldots, m\}$ and $\text{rad}_{q}(|t_{j}|) = \prod\limits_{i=1}^{n} p_{i}^{\mu_{ij}}$ for every  $j \in\{1, 2, \ldots, \ell\}$, where $\nu_{ij}, \mu_{ij} \geq 0$. Define 
\[
\mathcal{S} = \big\{ \langle ( \nu_{1j}, \nu_{2j}, \ldots, \nu_{nj} )\rangle_{\fq} \in\PG(\F_{q}^{n}) : j\in\{1, 2, \ldots, m\} \big\} \]
and
\[ \mathcal{T} = \big\{\langle (\mu_{1j}, \mu_{2j}, \ldots, \mu_{nj} ) \rangle_{\fq} \in\PG(\F_{q}^{n}) : j\in\{1, 2, \ldots, \ell\} \big\}
\]
to be the point sets in $\PG(\mathbb{F}_{q}^{n})$ associated with $S$ and $T$, respectively. We will say that the sets $S$ and $T$ are \textbf{geometric q-equivalent} if and only if there exists an element $\Psi \in\mathrm{PGL}(n,q)$ such that $\Psi(\mathcal{S}) = \mathcal{T}$. 
\end{definition}

By using the geometric description of sets in $\mathcal{T}_{k,q}$ provided in Theorem \ref{mainresult}, we prove that property of whether a finite subset of $\mathbb{Z}$ belongs to $\mathcal{T}_{k,q}$ is invariant under geometric $q$-equivalence. 

\begin{proposition} \label{th:propertygeometricequiv}
    Let $S=\{s_j\}_{j=1}^{m} \subset \Z \setminus \{0\}$ be a set of integers not containing a perfect $q^{th}$ power. Assume that $S \in \mathcal{T}_{k,q}$. Then every set $T = \{t_{j}\}_{j=1}^{\ell}$ that is geometric $q$-equivalent to $S$ belongs to $\mathcal{T}_{k,q}$. 
\end{proposition}

\begin{proof}
Let \( n' \) be the number of primes dividing \( \prod_{j=1}^{\ell} \rad(\lvert s_j\rvert) \). Since \( S \in \mathcal{T}_{k,q} \), by Theorem \ref{mainresult}, the set \( \mathcal{S}' \) of projective points associated with \( S \) forms a \( k \)-blocking set in \( \PG(\mathbb{F}_{q}^{n'}) \). Now, let \( p_1, \ldots, p_n \) be the primes dividing \( \prod_{j=1}^{m} \text{rad}_{q}(\lvert s_{j} \rvert) \times \prod_{j=1}^{\ell} \text{rad}_{q}(\lvert t_{j} \rvert) \), and let \( \mathcal{S} \subseteq \PG(\mathbb{F}_q^n) \) be the point set associated with \( S \) as in \Cref{def:geometricqequiv}. By construction, \( \mathcal{S} \) is obtained from \( \mathcal{S}' \) by adding \( n - n' \) zero components to the vectors representing the elements of \( \mathcal{S}' \). Therefore, the set \( \mathcal{S} \) is also a \( k \)-blocking set in \( \PG(\mathbb{F}_{q}^{n}) \) because \( n \geq n' \) (cf. \cite[Proposition 2.3]{blokhuis2011blocking}).

Let \( \mathcal{T} \) be the point set associated with \( T \) as in \Cref{def:geometricqequiv}. Since \( T \) is geometrically \( q \)-equivalent to \( S \), there exists an element \( \Psi \in \mathrm{PGL}(n,q) \) such that
\[
\Psi(\mathcal{S}) = \mathcal{T}.
\]
The property of being a \( k \)-blocking set is invariant under the action of \( \mathrm{PGL}(n,q) \) on \( \PG(\mathbb{F}_q^n) \), and thus \( \mathcal{T} \) is also a \( k \)-blocking set in \( \PG(\mathbb{F}_q^n) \). Let \( n^{\prime\prime} \) be the number of primes dividing \( \prod_{h=1}^{m} \text{rad}_{q}(\lvert t_{h} \rvert) \). The point set \( \mathcal{T}' \subseteq \PG(\mathbb{F}_q^{n''}) \) associated with \( T \) as in \Cref{ProjectiveIdentification} is obtained from \( \mathcal{T} \) by removing \( n - n'' \) zero components from the coordinates of the vectors representing the elements of \( \mathcal{T} \), corresponding to the primes in \( \{p_1, \ldots, p_n\} \) that do not divide \( \prod_{h=1}^{m} \text{rad}_{q}(\lvert t_{h} \rvert) \). It is straightforward to check that \( \mathcal{T}' \) is a \( k \)-blocking set in \( \PG(\mathbb{F}_q^{n''}) \).

Again, by Theorem \ref{mainresult}, we conclude that \( T \in \mathcal{T}_{k,q} \), and thus the assertion follows.
\end{proof}

An immediate consequence of the Propostion \ref{th:propertygeometricequiv} is that sets in $\mathcal{T}_{k,q}$ are also invariant under prime-wise exponentiation in the following sense.

\begin{corollary}
$\left( \text{Powers of primes in the factorization} \right)$  Let $p_1,\ldots,p_n$ be distinct primes and $S = \left\{\prod_{i=1}^{n} p_{i}^{\nu_{ij}}\right\}_{j=1}^{\ell}$ be a finite subset of integers . Assume that $S \in \mathcal{T}_{k,q}$. Then 
    \[S'=\left\{\prod_{i=1}^np_{i}^{b_i\nu_{ij}}\right\}_{j=1}^{\ell} \in \mathcal{T}_{k,q},\]
    for every $b_1,\ldots,b_{n} \in \F_q^{\times}$.
\end{corollary}

\begin{proof}
    The point sets in $\PG(\F_q^n)$ associated with $S$ and $S'$ are 
    \[
\mathcal{S} = \big\{ \langle ( \nu_{1j}, \nu_{2j}, \ldots, \nu_{nj} )\rangle_{\fq} \in\PG(\F_{q}^{n}) : j\in\{1, 2, \ldots, \ell\} \big\} \]
and
\[ \mathcal{S}' = \big\{\langle ( b_1\nu_{1j}, b_2\nu_{2j}, \ldots, b_n\nu_{nj} )\rangle_{\fq} \in\PG(\F_{q}^{n}) : j\in\{1, 2, \ldots, \ell\} \big\},
\]
respectively. The assertion follows by considering the element of $\mathrm{PGL}(n,q)$ induced by the diagonal matrix
\[
\begin{pmatrix}
    b_1 & 0 & \cdots & 0 \\
    0 & b_2 & \cdots & 0 \\
    \vdots &  & \ddots & \vdots \\
    0 & \cdots & 0 & b_n
\end{pmatrix}.
\]
\end{proof}

Now, we determine bounds on the size of the sets in $\mathcal{T}_{k,q}$. 
First, we will need the following classical bound on the size of a $k$-blocking set. 

\begin{proposition} [see \textnormal{\cite{bose1966characterization}}] \label{th:boundandclass}
    A $k$-blocking set of $\PG(\F_q^n)$ has at least $\frac{q^{k+1}-1}{q-1}$ points. In case of equality the blocking set is the point set of a $k$-space of $\PG(\F_q^n)$.
\end{proposition}

We immediately obtain the following result through combination of Proposition \ref{th:boundandclass} and Theorem \ref{mainresult}. This corollary also appears in \cite[Theorem 5]{skalba2004alternatives}.

\begin{proposition}
\label{lowerbound}
Let $S $ be a finite subset of integers not containing a perfect $q^{th}$ power. Assume that $S \in \mathcal{T}_{k,q}$. Then 
\[
\lvert S \rvert \geq \frac{q^{k+1}-1}{q-1} = q^{k} + q^{k-1} + \ldots + 1 .
\]
\end{proposition}

An interesting consequence of the Proposition \ref{lowerbound} is the following, which states that only way for a subset of smaller cardinality (than the lower bound above) to be in $\mathcal{T}_{k,q}$ is the trivial way by already containing a perfect $q^{th}$ power. 

\begin{corollary} [see \textnormal{\cite[Theorem 5]{skalba2004alternatives}}]
Let $S$ be a finite subset of integers with $|S| \leq q^{k} + q^{k-1} + \ldots + q$. Then, $S \in\mathcal{T}_{k,q}$ if and only if $S$ contains a perfect $q^{th}$ power. 
\end{corollary}

In addition to obtaining lower bounds on their cardinalities, another advantage of studying sets in $\mathcal{T}_{k, q}$ geometrically is that we can completely outline the factorization pattern of elements of the set $S \in\mathcal{T}_{k,q}$ that attain the lower bound. 


\begin{proposition} \label{th:classification}
    Let $S=\{s_j\}_{j=1}^{\ell}$ be set of integers not containing a perfect $q^{th}$ power with $\ell=q^k+\cdots+q+1$. 
    Let $n$ be the number of distinct primes that divide the $\prod_{j=1}^{\ell}\rad(\lvert s_j\rvert)$. Let $\mathcal{S} \subseteq \PG(\F_q^n)$ be the set of projective points associated with $S$ as in \ref{ProjectiveIdentification}.
    Then $S \in \mathcal{T}_{k,q}$ if and only if $\mathcal{S}$ is a $k$-space of $\PG(\F_{q}^{n})$. In such a case, $S$ is geometrically $q$-equivalent to the set
        \begin{multline*} \label{lineexample}
        \overline{S}=\{\overline{p}_1\overline{p}_2^{\alpha_2}\cdots\overline{p}_{k+1}^{\alpha_{k+1}} \colon \alpha_i \in \F_q\} \\ \bigcup \{\overline{p}_2\overline{p}_3^{\alpha_3}\cdots\overline{p}_{k+1}^{\alpha_{k+1}} \colon  \alpha_i \in \F_q\} \cdots \\
         \cdots
         \bigcup \{\overline{p}_k\overline{p}_{k+1}^{\alpha_{k+1}} \colon  \alpha_i \in \F_q\} 
         \bigcup \{\overline{p}_{k+1}\}
        \end{multline*}
    for every $k+1$ distinct primes $\overline{p}_1,\ldots,\overline{p}_{k+1}$. 
\end{proposition}

\begin{proof}
By using \Cref{mainresult}, we know that, $S \in \mathcal{T}_{k,q}$ if and only if 
$\mathcal{S}$ is a $k$-blocking set of $\PG(\F_{q}^{n})$. Moreover, by hypothesis we have 
\[
q^k+\cdots +q +1 = \lvert S \rvert \geq \lvert \mathcal{S} \rvert \geq q^k+\cdots +q +1.
\]
where the last inequality follows from \Cref{th:boundandclass}. So this is equivalent to say that $\mathcal{S}$ is a $k$-blocking set of $\PG(\F_q^n)$ having size $q^k+\cdots+q+1$, or in other words, by \Cref{th:boundandclass}, $\mathcal{S}$ is a $k$-space. This proves the first part of the assertion. \\
Now, assume that $p_1,\ldots,p_n$ are all the distinct primes that divide the $\prod_{j=1}^{\ell}\rad(\lvert s_j\rvert)$. Let $\overline{p}_1,\ldots,\overline{p}_{k+1}$ be $k+1$ distinct primes that does not divide the $\prod_{j=1}^{\ell}\rad(\lvert s_j\rvert)$ and consider $\overline{S}$ as in the statement. Then $\overline{p}_1,\ldots,\overline{p}_{k+1},p_1,\ldots,p_n$ are all the distinct primes that divides $\prod_{\overline{s} \in \overline{S}}\rad(\lvert \overline{s}\rvert)\prod_{j=1}^{\ell}\rad(\lvert s_j\rvert)$. The set of points associated with $\overline{S}$ with respect to $\overline{p}_1,\ldots,\overline{p}_{k+1},p_1,\ldots,p_n$ is $\PG(U) \subseteq \PG(\F_q^{n+k+1})$, where $U$ is the $\F_q$-vector space generated by
\[
(1,0,\ldots,0), (0,1,0,\ldots,0),\ldots,(0,\ldots,0,\underbrace{1}_{k+1},0\ldots ) \in \F_q^{n+k+1}. 
\] 
On the other hand, the set of points associated with $S$ with respect to $\overline{p}_1,\ldots,\overline{p}_{k+1},p_1,\ldots,p_n$ is $\PG(W) \subseteq \PG(\F_q^{n+k+1})$, where $W$ is the $\F_q$-vector space of dimension $k+1$ contained in $\{0\}^{k+1} \times \F_q^n$. Since $\PG(U)$ and $\PG(W)$ are $k$-spaces of $\PG(\F_q^{k+1+n})$, there exists an element of $\PG(k+1+n,q)$ mapping $\PG(U)$ in $\PG(W)$. Hence $S$ is geometric $q$-equivalent to $\overline{S}$.
The assertion follows from $(2)$ of \Cref{cor:indendenceprimes}.
\end{proof}

Now we present two examples that demonstrate how \Cref{th:classification} can be employed to construct minimum-size sets in $\mathcal{T}_{k,q}$ and to establish when sets of size $q^{k-1}+\cdots+q+1$ do not belong to $\mathcal{T}_{k,q}$.

\begin{example}
Let $q=3$, and consider $k=2$. As proved in \Cref{th:classification}, the ``standard" set in $\mathcal{T}_{2,3}$ having minimum size $q^2+q+1=13$ is given by 
\[
\{p_1,p_2,p_1p_2,p_1p_2^2,p_3,p_1p_3,p_1p_3^2,p_2p_3,p_2p_3^2,p_1p_2p_3,p_1p_2^2p_3,p_1p_2p_3^2,p_1p_2^2,p_3^2\},
\]
where $p_1,p_2,p_3$ are any distinct primes. However, the classification in \Cref{th:classification} provides many more examples of sets in $\mathcal{T}_{2,3}$ having minimum size 13. For instance, let us consider the projective space $\PG(\F_3^4)$. We can consider the $2$-space $\PG(W)$, where $W=\langle (1,0,0,0), (0,1,1,1), (0,1,0,2) \rangle_{\F_q}$. 
So,
\[
\PG(W)=\left\{ \begin{array}{cccc}
     \langle (1,0, 0,0) \rangle_{\fq}, & \langle (0,1,1,1)\rangle_{\fq}, &\langle (0, 1,0, 2) \rangle_{\fq}, & \langle (1,1,1,1)\rangle_{\fq}, \\
     \langle (1, 2, 2,2) \rangle_{\fq},  &
\langle (1,1, 0,2) \rangle_{\fq}, & \langle (1,2,0,1)\rangle_{\fq}, & \langle (0,2,1,0)\rangle_{\fq}, \\ \langle (0, 0,1, 2) \rangle_{\fq}, &\langle (1, 2, 1,0) \rangle_{\fq}, &
\langle (1,0, 1,2) \rangle_{\fq}, & \langle (1,0,2,1)\rangle_{\fq},  \\
\langle (1,1,2,0)\rangle_{\fq} &&&
\end{array}\right\} \subseteq \PG(\F_3^4).
\]
Therefore, by using \Cref{th:classification}, for any distinct primes $p_1,p_2,p_3,p_4$, the set
\[
S=\{p_1,p_2p_3p_4,p_2p_4^2,p_1p_2p_3p_4,p_1p_2^2p_3^2p_4^2,p_1p_2p_4^2,p_1p_2^2p_4,p_2^2p_3,p_3p_4^2,p_1p_2^2p_3,p_1p_3p_4^2,p_1p_3^2p_4,p_1p_2p_3^2\}
\] 
belongs to $\mathcal{T}_{2,3}$.
\end{example}

Also, in the case where $\lvert S \rvert \geq q^{k} + q^{k-1} + \ldots + q + 1$, the geometric connection between elements of $\mathcal{T}_{k,q}$ and a $k$-blocking set can be employed to show that a given set is not in $\mathcal{T}_{k,q}$. Consider the following example.

\begin{example} \label{ex:moreprimes}
Suppose we want to investigate whether the set \[
    \begin{array}{rl}
          S& =\{2,3,5,6,
          7,10,14,15,
          20,35,42,50,180\} \\
         & =\{2,3,5,2\cdot 3,7,2 \cdot 5,2\cdot 7, 3 \cdot 5,2^2 \cdot 5, 3 \cdot 7,2 \cdot 3 \cdot 7, 2 \cdot 5^2, 2^2 \cdot 3^2 \cdot 5\},
    \end{array}
   \]
   belongs to $\mathcal{T}_{2,3}$. The set of all primes that divide an element of $S$ is $\{2, 3,5,7\}$. The point set associated with $S$ as in \Cref{ProjectiveIdentification} is 
    \[
\mathcal{S}= \left\{ \begin{array}{cccc}
     \langle (1,0, 0,0) \rangle_{\fq}, & \langle (0,1,0,0)\rangle_{\fq}, &\langle (0, 0,1, 0) \rangle_{\fq}, &
     \langle (1, 1, 0,0) \rangle_{\fq},  \\  
\langle (0,0, 0,1) \rangle_{\fq}, & \langle (1,0,1,0)\rangle_{\fq}, & \langle (1, 0,0, 1) \rangle_{\fq}, &\langle (0, 1, 1,0) \rangle_{\fq},\\
\langle (2,0, 0,1) \rangle_{\fq}, & \langle (0,0,1,1)\rangle_{\fq}, & \langle (1, 1,0, 1) \rangle_{\fq}, &\langle (1, 0,2,0) \rangle_{\fq}, \\
\langle (2, 2,1,0) \rangle_{\fq} &&&
\end{array}\right\} \subseteq \PG(\F_3^4).
\]
Note that $\lvert S \rvert=\lvert \mathcal{S} \rvert = 13$. Therefore, by \Cref{th:classification}, we get that $S \in \mathcal{T}_{2,3}$ if and only if $\mathcal{S}$ is a plane of $\PG(\F_3^4)$. Observe that \[
\langle (1,0, 0,0) \rangle_{\fq}, \langle (0,1,0,0)\rangle_{\fq}, \langle (0, 0,1, 0) \rangle_{\fq},  
\langle (0,0, 0,1) \rangle_{\fq} \in \mathcal{S},\] implying that $\mathcal{S}$ cannot be a plane. Therefore $S \notin \mathcal{T}_{2,3}$
\end{example}

Proposition \ref{th:classification} shows that the smallest non-trivial set $S$ in the collection $\mathcal{T}_{k,q}$ is of size $\frac{q^{k+1} - 1}{q^{k} - 1}$. Note that for every non-trivial $S \in \mathcal{T}_{k,q}$ and any integer $a$, we also have $S \cup \{a\} \in\mathcal{T}_{k,q}$. In general, for every superset $T$ of $S$, $T \in\mathcal{T}_{k,q}$. Therefore, to avoid redundancies in classification, we introduce the following definition. 

\begin{definition}
A set $S \in \mathcal{T}_{k,q}$ will be called \textbf{minimal} if there does not exist a  proper subset $T \subset S$ such that $T \in \mathcal{T}_{k,q}$. 
\end{definition}

A set $S \in\mathcal{T}_{k,q}$ that achieves the lower bound $\lvert S \rvert = \frac{q^{k+1} - 1}{q - 1}$ and does not contain a perfect $q^{th}$ power is clearly minimal as a consequence of Proposition \ref{lowerbound}. Interestingly, there are no minimal sets in $\mathcal{T}_{k,q}$ of cardinality $\frac{q^{k+1} - 1}{q - 1} + 1$ or $\frac{q^{k+1} - 1}{q - 1} + 2$. As before, the next cardinality of a minimal set in $\mathcal{T}_{k,q}$ will be implied by the corresponding result about $k$-blocking sets. 

\begin{definition}
    Let $B$ be a set of points in $\PG(\F_q^n)$ and $\Lambda$ be a subspace such that $\Lambda \cap B = \emptyset$.  The \emph{cone} with vertex $\Lambda$ and base $B$ is the union of $\Lambda$ and the subspaces generated by $\Lambda$ and $P$, with $P \in B$.
\end{definition}

\begin{proposition} 
[see \cite{heim1996blockierende}] 
\label{th:secondsmallest}
    Let $n, k$ be natural number with $n > k$ and $n \geq 3$ and let $\mathcal{S}$ be a $k$-blocking set of $\PG(\F_q^n)$ not containing a $k$-space in $\PG(\F_q^n)$. Then, \[\lvert \mathcal{S} \rvert \geq \frac{q^{k+1}-1}{q-1}+q^{k-1}\frac{q+1}{2}.\]  Furthermore, the equality is achieved above if and only if $\mathcal{S}$ is a cone with vertex a $(k-2)$-space $\Lambda$ and as base a blocking set $\overline{\mathcal{S}}$ of a plane $\Sigma$ such that $\Lambda \cap \Sigma = \emptyset$ and $\lvert \overline{\mathcal{S}} \rvert = 3 \frac{q+1}{2}$.
\end{proposition}

By making use of the geometric characterization of sets in $\mathcal{T}_{k,q}$ provided by Theorem \ref{mainresult}, and by employing the bounds on the size of $k$-blocking sets recalled in Proposition \ref{th:secondsmallest}, we are able to prove the following bounds and existence results for sets in $\mathcal{T}_{k,q}$.

\begin{proposition}
    Let $S$ be set of integers not containing a perfect $q^{th}$ power such that $S \in \mathcal{T}_{k,q}$ and $S$ is minimal. 
    Then, \[\lvert S \rvert >q^k+ q^{k-1} \ldots +q+1 \text{   if and only if   }
    \lvert S \rvert \geq \frac{q^{k+1}-1}{q-1}+q^{k-1}\frac{q+1}{2}.
    \]
In other words, minimal $S$ in $\mathcal{T}_{k,q}$ with cardinality in the interval 
\begin{equation} \label{eq:secondinterval}
\left(\frac{q^{k+1}-1}{q-1},  \frac{q^{k+1}-1}{q-1}+q^{k-1}\cdot \frac{q+1}{2} \right)
\end{equation}
does not exist and any $S \in\mathcal{T}_{k,q}$ with cardinality in the interval as in \eqref{eq:secondinterval} is not minimal and contains a subset that is geometrically $q$-equivalent to that in Proposition \ref{th:classification}.
\end{proposition}

\begin{proof}
    Let $\mathcal{S} \subseteq \PG(\F_q^n)$ be the set of projective points associated with $S$, as in \ref{ProjectiveIdentification}. Since $S \in \mathcal{T}_{k,q}$, by \Cref{mainresult}, we have that $\mathcal{S}$ is a $k$-blocking set of $\PG(\F_q^n)$. Now, the set $S$ is minimal, from which it is easy to check that $\lvert S \rvert = \lvert \mathcal{S} \rvert$, and any proper subset of $\mathcal{S}$ is no longer a $k$-blocking set. Therefore, since the size of $S$ is greater than $q^k+ q^{k-1} \ldots +q+1$, we get that $\mathcal{S}$ cannot contain a $k$-space of $\PG(\F_q^n)$. Hence, by using \Cref{th:secondsmallest}, we get that
    \[
    \lvert S \rvert =\lvert \mathcal{S} \rvert \geq \frac{q^{k+1}-1}{q-1}+q^{k-1}\frac{q+1}{2}
    \]
    that proves our assertion.
\end{proof}

\subsection{Elements of $\mathcal{T}_{k,q}$ with Second Smallest Cardinality}
Using the equality case of Proposition \ref{th:secondsmallest}, we can construct minimal sets in $\mathcal{T}_{k,q}$ with cardinality $\frac{q^{k+1}-1}{q-1} + q^{k-1} \frac{q+1}{2}$, up to geometric $q$-equivalence, for every $k \geq 2$ and odd prime number $q$.

Let $Q$ be the set of quadratic residues of $\mathbb{F}_q$, i.e.
$Q:= \{a^2 \colon a \in \mathbb{F}_q^*\}$ and let $Q_{0} := Q \cup \{0\}$. Since $q$ is an odd prime, $\lvert Q_{0} \rvert = \frac{q+1}{2}$

\begin{proposition} [see \textnormal{\cite[Lemma 13.6 (i)]{hirschfeld1998projective}}] \label{prop:constructionprojectivetriangle}
Let $q$ be an odd prime. The point set
\begin{small}
\[
\overline{\mathcal{S}}=\{\langle (0, 1, -s) \rangle,\langle(-s, 0, 1)\rangle,\langle (1, -s, 0) \rangle \colon s \in Q_0\}
\]
\end{small} 
is a blocking set of $\PG(\F_q^3)$ having size $3\frac{q+1}{2}$.
\end{proposition}

Consider $\overline{\mathcal{S}}\subseteq \PG(\F_q^3)$ as in \Cref{prop:constructionprojectivetriangle} with $n\geq k+2$. We can embed $\Sigma = \PG(\F_q^3)$ in $\PG(\F_q^n)$ by letting the last coordinates equals to 0. Consider the $\F_q$-subspace $W$ generated by \[
(0,0,0,1,0,\ldots,0),(0,0,0,0,1,0,\ldots,0),\ldots, (0,\ldots,0,\underbrace{1}_{k+2},0,\ldots,0) 
\] 
and let $\Lambda= \PG(W) \subseteq \PG(\F_q^n)$. Clearly, $\Lambda$ is a $(k-2)$-space of $\PG(\F_q^n)$ and $\Sigma \cap \Lambda= \emptyset$. Therefore, the cone with vertex $\Lambda$ and basis $\overline{\mathcal{S}} \subseteq \Sigma$ is the point set

\begin{multline} \label{eq:blockingsetsecond}
\mathcal{S} = \Big\{\langle (0, 1, -s,\alpha_4,\ldots,\alpha_{k+2},0,\ldots,0 ) \rangle, \langle(-s, 0, 1,\alpha_4,\ldots,\alpha_{k+2},0,\ldots,0 )\rangle, \\ \langle (1, -s, 0,\alpha_4,\ldots,\alpha_{k+2},0,\ldots,0 ) \rangle \colon s \in Q_0 \mbox{ and }\alpha_i \in \mathbb{F}_q\Big\}  \cup\  \PG(W),
\end{multline}

and it is a minimal $k$-blocking set of $\PG(\F_q^n)$ having size $\frac{q^{k+1}-1}{q-1}+q^{k-1}\frac{q+1}{2}$.

As a consequence, \Cref{mainresult}, together with the point set $\mathcal{S}$, provides constructions of minimal sets in $\mathcal{T}_{k,q}$ having the second smallest cardinality $\frac{q^{k+1}-1}{q-1}+q^{k-1}\frac{q+1}{2}$, for every $k \geq 2$ and odd prime $q$.

\begin{theorem}
For any odd prime $q$ and $k \geq 2$, the set
\[
\begin{array}{rl}
        S=& \{p_2p_3^{q-s}p_4^{\alpha_4}p_5^{\alpha_5}\cdots p_{k+2}^{\alpha_{k+2}} \colon s \in Q_0, \alpha_i \in \F_q\}    \\
        & \cup \{p_1^{q-s}p_3p_4^{\alpha_4}p_5^{\alpha_5}\cdots p_{k+2}^{\alpha_{k+2}}\colon s \in Q_0, \alpha_i \in \F_q \} \\
        & \cup \{p_1p_2^{q-s}p_4^{\alpha_4}p_5^{\alpha_5}\cdots p_{k+2}^{\alpha_{k+2}} \colon s \in Q_0, \alpha_i \in \F_q\} \\
&\cup \{p_4p_5^{\alpha_5}\cdots p_{k+2}^{\alpha_{k+2}} \colon \alpha_i \in \F_q\} \\
        & \cup \{p_5p_6^{\alpha_6}\cdots p_{k+2}^{\alpha_{k+2}} \colon  \alpha_i \in \F_q\} \\
        & \cdots \\
        & \cup \{p_{k+1}p_{k+2}^{\alpha_{k+2}} \colon  \alpha_i \in \F_q\} \\
        & \cup \{p_{k+2}\},
        \end{array}
\]
where $p_1,\ldots,p_{k+2}$ are $k+2$ distinct primes, is a minimal set in $\mathcal{T}_{k,q}$ not containing a perfect $q^{th}$ power and having size $\frac{q^{k+1}-1}{q-1}+q^{k-1}\frac{q+1}{2}$.
\end{theorem}

\section*{Acknowledgement}\footnotesize
The authors are grateful to the reviewer for their invaluable comments and suggestions, which enhanced the exposition in this article. The second author research was partially supported by the Italian National Group for Algebraic and Geometric Structures and their Applications (GNSAGA - INdAM) and by the INdAM - GNSAGA Project \emph{Tensors over finite fields and their applications}, number E53C23001670001 and by Bando Galileo 2024 – G24-216.

\bibliographystyle{abbrv}
\bibliography{blocking}

\medskip

Bhawesh Mishra \\
Department of Mathematical Sciences, \\
384 Dunn Hall, University of Memphis,\\
Memphis, TN 38107, USA \\
{{\em bmishra1@memphis.edu}}

\medskip

Paolo Santonastaso\\
Dipartimento di Matematica e Fisica,\\
Universit\`a degli Studi della Campania ``Luigi Vanvitelli'',\\
I--\,81100 Caserta, Italy\\
{{\em paolo.santonastaso@unicampania.it}}\\
Dipartimento di Meccanica, Matematica e Management, \\
Politecnico di Bari, \\
Via Orabona 4, \\
70125 Bari, Italy \\
{{\em paolo.santonastaso@poliba.it}}

\end{document}